\documentclass[12pt]{amsart}

\usepackage{amsthm,amssymb,amstext,amscd,amsfonts,amsbsy,amsxtra,latexsym,amsmath,mathtools}
\usepackage{fullpage}
\usepackage[english]{babel}
\usepackage[latin1]{inputenc}
\usepackage{verbatim}
\usepackage{graphicx}
\usepackage[all]{xy}
\numberwithin{figure}{section}
\usepackage{amsmath}
\usepackage{amssymb} 
\usepackage{ifthen}

\newcommand{\bbZ}{{\mathbb Z}}
\newcommand{\bbQ}{{\mathbb Q}}

\newcommand{\bbL}{{\mathbb L}}

\newcommand{\C}{{\mathbb C}}
\newcommand{\N}{{\mathbb N}}
\newcommand{\Z}{{\mathbb Z}}
\newcommand{\Q}{{\mathbb Q}}

\numberwithin{equation}{section}
\newtheorem{theorem}{\textbf{Theorem}}
\newtheorem{lemma}[theorem]{\textbf{Lemma}}

\newtheorem{corollary}[theorem]{\textbf{Corollary}}
\numberwithin{theorem}{section}
\newtheorem*{conjecture}{Conjecture}
\newtheorem*{question}{Question}
\newtheorem*{definition}{Definition}

\theoremstyle{remark}
\newtheorem*{remark}{Remark}

\renewenvironment{proof}[1][Proof]{\begin{trivlist}
\item[\hskip \labelsep {\bfseries #1:}]}{\qed\end{trivlist}}

\newboolean{inclcom}
\setboolean{inclcom}{false}
\newcommand{\condcomment}[2]{\ifthenelse{#1}{#2}{}}

\begin{document}  

\author[P. Guerzhoy]{Pavel Guerzhoy}
\author[Z. Kent]{Zachary A. Kent}
\address{ 
	Department of Mathematics,
	University of Hawaii, 
	2565 McCarthy Mall, 
	Honolulu, HI,  96822-2273 
}
\email{pavel@math.hawaii.edu}
\email{zach@math.hawaii.edu}
\author[L. Rolen]{Larry Rolen}
\address{
	Mathematical Institute,
	University of Cologne,
	Weyertal 86-90,
	50931 Cologne, Germany
}
\email{lrolen@math.uni-koeln.de}
\title[Congruences for quantum modular forms]{Congruences for Taylor expansions of quantum modular forms}
\thanks{The research of first author is supported by the Simons Foundation Collaboration Grant.
This research was conducted while the first author was a guest at MPIM, and he is grateful
to the Institute for making this research possible. The first author thanks Don Zagier for
an elucidating communication. The third author thanks the University of Cologne and the DFG for their generous support via the University of Cologne postdoc grant DFG Grant D-72133-G-403-151001011, funded under the Institutional Strategy of the University of Cologne within the German Excellence Initiative. The authors are also grateful to Kathrin Bringmann,  Byungchan Kim, and Armin Straub for useful comments which improved the paper. }

\begin{abstract}
Recently, a beautiful paper of Andrews and Sellers has established linear congruences for the Fishburn numbers modulo an infinite set of primes. Since then, a number of authors have proven refined results, for example, extending all of these congruences to arbitrary powers of the primes involved. Here, we take a different perspective and explain the general theory of such congruences in the context of an important class of quantum modular forms. As one example, we obtain an infinite series of combinatorial sequences connected to the ``half-derivatives" of the Andrews-Gordon functions  and with Kashaev's invariant on $(2m+1,2)$ torus knots, and we prove conditions under which the sequences satisfy linear congruences modulo at least $50\%$ of primes.
\end{abstract}

\maketitle

\section{Introduction and statement of results}
In his seminal 2010 Clay lecture, Zagier defined a new class of functions with certain automorphic properties called ``quantum modular forms'' \cite{Zagier-Clay}. Roughly speaking, these are complex valued functions defined on the rational numbers which have modular transformations modulo ``nice'' functions. Although the definition is (intentionally) vague, Zagier gave a handful of motivating examples to serve as prototypes. For example, he defined quantum modular forms related to Maass cusp forms attached to Hecke characters of real quadratic fields, as studied by Andrews, Dyson and Hickerson \cite{AndrewsDysonHickerson} and Cohen \cite{Cohen}, and he gave examples related to sums over quadratic polynomials and non-holomorphic Eichler integrals. More precisely, Zagier made the following definition.
\begin{definition} A \textbf{quantum modular form} is a function $f\colon\mathbb{P}^1(\Q)\rightarrow\C$ for which $f(x)-f|_k\gamma(x)$ is ``suitably nice.''
\end{definition}
Here $|_k$ is the usual Petersson slash operator and ``suitably nice'' means that the obstruction to modularity satisfies an appropriate analyticity condition, e.g. $\mathcal{C}^k,\mathcal{C}^{\infty}$, etc. One of the most striking examples of a quantum modular form is described in Zagier's paper on Vassiliev invariants \cite{Zagier}, in which he studies Kontsevich's function $F(q)$ given by

\begin{equation}F(q):=\sum_{n=0}^{\infty}(q;q)_n,\end{equation}
where $(a,q)_n:=\prod_{j=0}^{n-1}(1-aq^j)$.

This function does not converge on any open subset of $\C$, but converges as a finite sum for $q$ any root of unity. Zagier's study of $F$ depends on the ``sum of tails'' identity 
\begin{equation}
\label{sumoftails}
\displaystyle\sum_{n\geq0}\left(\eta(q)-q^{\frac1{24}}\left(q;q\right)_n\right)=\eta(q)D\left(q\right)-\frac12\sum_{n\geq1}n\chi_{12}(n)q^{\frac{n^2-1}{24}},
\end{equation}
\noindent
where $\eta(q):=q^{\frac1{24}}(q;q)_{\infty}$, $D(q):=-\frac12+\sum_{n\geq1}\frac{q^n}{1-q^n},$ and $\chi_{12}(n):=\left(\frac{12}{n}\right)$. Recalling the identity $(q;q)_{\infty}=\sum_{n\geq1}\chi_{12}(n)q^{\frac{n^2-1}{24}}$, we find that the last term on the right hand side of (\ref{sumoftails}) is essentially a ``half-derivative'' of $\eta$.
Zagier further observed that $\eta(q)$ and $\eta(q)D(q)$ vanish to infinite order as $q$ approaches a root of unity. Thus, we have what Zagier terms a ``strange identity'' of the shape
\begin{equation}\label{KontsevichStrange}F(q)``="-\frac12\sum_{n\geq1}n\chi_{12}(n)q^{\frac{n^2-1}{24}}.\end{equation}
Although the left and right hand sides of (\ref{KontsevichStrange}) do not ever converge simultaneously, as we will see below, this identity can be interpreted as saying that there is an equality between asymptotic expansions when $q=e^{-t}$ as $t\rightarrow0^+$ (a similar statement holds as $q$ approaches other roots of unity as well).

Since \cite{Zagier-Clay}, there has been an explosion of research aimed at constructing examples of quantum modular forms related to Eichler integrals, extending the initial applications to knot invariants and quantum invariants of 3-manifolds in \cite{LawrenceZagier} and \cite{Zagier}. For instance, such quantum modular forms are closely tied to surprising identities relating generating functions of ranks, cranks, and unimodal sequences \cite{Folsom-Ono-Rhoades}, are related to probabilities of integer partition statistics \cite{NgoRhoades}, and arise in the study of negative index Jacobi forms and Kac-Wakimoto characters \cite{Bringmann-Creutzig-Rolen}. For further examples, see also \cite{Susie,BrysonOnoPitmanRhoades,BringmannRolenRadial,Costa,HikamiDecompWRT,HikamiLovejoy,LiNgoRhoades,RolenSchneieder}. 

The general theory of this class of quantum modular forms was further elucidated by Choi, Lim, and Rhoades in \cite{ChoiLimRhoades}, and from a different perspective by Bringmann and the third author in \cite{BringmannRolenConference}, where the space of ``Eichler quantum modular forms" was defined. In particular, for each half-integral weight cusp form there is an associated Eichler integral with quantum modular properties. At the end of \cite{BringmannRolenConference}, a program was laid out to study the general properties of these quantum modular forms. In particular, one of the fundamental problems in the theory was identified to be the determination of the arithmetic properties of such forms. This problem was inspired by recent work of Andrews and Sellers \cite{AndrewsSellers} in which they studied the congruence properties of the Fishburn numbers defined by
\[\sum_{n\geq0}\xi(n)q^n:=\sum_{n\geq0}\left(1-q;1-q\right)_n.\] 

These numbers are important in combinatorics and knot theory and in particular $\xi(n)$ enumerates the number of linearly independent Vassiliev invariants of degree $n$ \cite{Zagier}. For related results on the conbinatorics and asymptotic analysis of the Fishburn numbers and related sequences, see also \cite{AndrewsJelinek,Bousquet-Melou,BringmannLiRhoades,Chmutov,Claesson,Fishburn1,Jelinek,Levande,Stoimenow}. Utilizing beautiful and clever manipulations, Andrews and Sellers proved an infinite family of congruences for the Fishburn numbers, and these results were extended in several directions by Garvan \cite{Garvan}, Straub \cite{Straub}, and Ahlgren and Kim \cite{AhlgrenKim}. For example, the work of Andrews and Sellers and Straub implies that for all $A,n\in\N$, we have

\[\xi\left(5^An-1\right)\equiv\xi\left(5^An-2\right)\equiv0\pmod{5^A},\]
\[\xi\left(7^An-1\right)\equiv0\pmod{7^A},\]
\[\xi\left(11^An-1\right)\equiv\xi\left(11^An-2\right)\equiv\xi\left(11^An-3\right)\equiv0\pmod{11^A}.\]

In the context of quantum modular forms, the Fishburn numbers are the coefficients of the power series expansion at the root of unity $q=1$ of the Eichler quantum modular form associated to $\eta$. In light of this connection,
Garvan and Straub posed the following natural question.
\begin{question}[Garvan, Straub]
What is the general theory of such congruences for more general quantum modular forms?
\end{question}

In this paper, we answer this question for a class of quantum modular forms. In particular, we consider quantum modular forms which arise as Eichler integrals of unary theta series. That is, we consider the quantum modular forms associated to theta series of the form 
\[\sum_{n\in\Z}\chi(n)n^{\nu}q^{\frac{n^2-a^2}b},\]
where $\nu\in\{0,1\}$, $\chi$ is a periodic sequence of mean value zero, and $a,b\in\Z$. For the sake of definiteness, we will focus our attention on the case when $\nu=0$, i.e., when the unary theta series has weight $1/2$. We will shortly see how to define a ``Fishburn-type'' sequence associated to any such unary theta series via a series expansion as $q\rightarrow1$. We remark in passing that the analogous study of congruences for general modular forms of half-integral weight does not make sense as we do not know any algebraicity results about the associated power series coefficients.

Before stating our main result, we define a condition on the periodic sequences $\chi$ under consideration. This condition arises very naturally from the perspective of $p$-adic measures below, and is also satisfied by our canonical family of examples arising from knot invariants (see Section \ref{ExampleS}). Specifically, our condition on $\chi$ is as follows, where
\[\psi_u(x):=\begin{cases}
1&\text{ if }x\equiv u \pmod M,\\
0&\text{ otherwise},
\end{cases}
\]
and $\psi_{u,v}(x):=\psi_u(x)-\psi_v(x)$.

\begin{definition}
Let $\chi\colon\Z\rightarrow\{0,1,-1\}$ be a periodic function with period $M$ and mean value zero.
We say that $\chi$ is a {\bf good function} if it is a sum of functions of the form $\psi_{u,v}$ such that there exist natural numbers $C:=C(u,v)$ with $(C,M)=1$ and $uC\equiv v\pmod M$. 
\end{definition}
\begin{remark}
The condition for $\chi$ to be good is trivially satisfied whenever $\chi$ is supported on integers coprime to $M$. In particular, any quadratic Dirichlet character is good. 
\end{remark}
To describe our main result, we first consider for any $a,b\in\Z$ and any periodic sequence $\chi$ the following partial theta function:
 \[P_{a,b,\chi}(q):=\sum_{n\geq0}n\chi(n)q^{\frac{n^2-a^2}{b}}.\]
Note that this is essentially a ``half-derivative'' of the series $\sum_{n}\chi(n)q^{\frac{n^2-a^2}{b}}$. From the general theory outlined in \cite{BringmannRolenConference}, $\chi$ has mean value zero in particular whenever $P_{a,b,\chi}$ is a half-derivative of a modular form which is cuspidal at $q=1$. From now on, we will always assume that $P_{a,b,\chi}\in\Z[[q]]$. As we will see in the proof of Theorem \ref{xinumberlval}, the mean value zero property of $\chi$ implies that there is an asymptotic expansion 
 \[P_{a,b,\chi}\left(e^{-t}\right)\sim\sum_{n\geq0}\alpha_{a,b,\chi}(n)t^n\]
for some $\alpha_{a,b,\chi}(n)\in\C$ as $t\rightarrow0^+$. Thus, we may define coefficients $H_{a,b,\chi}(n)$ by the relation
\[P_{a,b,\chi}\left(e^{-t}\right)\sim\sum_{n\geq0}H_{a,b,\chi}(n)\left(1-e^{-t}\right)^n.\]
It is easily seen that the coefficients $H_{a,b,\chi}(n)$ are defined by finite recursions in terms of $\alpha_{a,b,\chi}(n)$. Moreover, the explicit formulas for the coefficients in Theorem \ref{xinumberlval} show that $H_{a,b,\chi}(n)\in\Q$ for all $n$. We will see below that in important examples arising from knot theory, they may also be defined combinatorially in a manner similar to the definition of the Fishburn numbers above.
 
 Our main result is then the following, where we define $\beta_p=\beta$ to be the coefficient of $p^1$ in the base $p$ expansion of $-\frac{a^2}b$.
 \begin{theorem}\label{mainthm}
Let $\chi$ be a good function and suppose that $a,b\in\Z$ are chosen so that $P_{a,b,\chi}\in\Z[[q]]$. Then for any prime $p$ not dividing $b$, the following are true.
\begin{enumerate}
\item If $B$ is a positive integer such that
		\[
			\left(\frac{a^2-b}{p}\right) =-1, \ \left(\frac{a^2-2b}{p}\right) =-1, \dots, \ \left(\frac{a^2-Bb}{p}\right) =-1,\]
			then for all $A,n\in\N$, we have
		\[
			H_{a,b,\chi}\left(p^A n - B\right) \equiv 0 \pmod {p^A}.
		\]
\item  If $\beta\neq (p-1)$ and $B$ is a positive integer such that
		\[
			\left(\frac{a^2-b}{p}\right) \neq 1, \ \left(\frac{a^2-2b}{p}\right) \neq 1, \dots, \ \left(\frac{a^2-Bb}{p}\right) \neq 1,\]
			then for all $A,n\in\N$, we have
		\[
			H_{a,b,\chi}\left(p^A n - B\right) \equiv 0 \pmod {p^A}.
		\]
\end{enumerate}
\end{theorem}
\noindent
{\it Three remarks.}

\noindent 1) We will see in Section \ref{ExampleS} that Theorem \ref{mainthm} implies all known linear congruences for $\xi(n)$.

\smallskip
\noindent 2) We will give a stronger version of Theorem \ref{mainthm} in Theorem 3.3 below.

\smallskip

\noindent
3) As asked by Andrews and Sellers and Straub in the case of $\xi(n)$, it is an interesting question to determine a converse of Theorem \ref{mainthm} classifying linear congruences for the sequences $H_{a,b,\chi}(n)$. This seems difficult to prove using our techniques. From the perspective of $p$-adic measures outlined in Section \ref{Prelim}, Theorem \ref{padic} aligns with the intuitive notion that the ``integral'' of a function which is $p$-adically small over a compact set is small. However, the difficulty in proving a converse result lies in the fact that the integral of a large function can still be small.
\smallskip

In particular, we have the following corollary, which we will show in Section \ref{ExampleS} implies that infinitely many of the sequences arising in our family of examples from knot theory satisfy linear congruences modulo at least $50\%$ of primes.
\begin{corollary}
Let $\chi$ be a good function and suppose that $a,b\in\Z$ are chosen so that $P_{a,b,\chi}\in\Z[[q]]$, and that $a^2-b\not\equiv 3\pmod4$ and such that $a^2-b$ is not a square. Then there is a linear congruence for $H_{a,b,\chi}(n)$ modulo at least $50\%$ of primes.
\end{corollary}
\begin{remark}
We will see in Section \ref{ExampleS} that the condition that $a^2-b$ is not a square is necessary, where we will give an example in which our theorem does not guarantee any congruences.
\end{remark}

We now highlight an important situation when the coefficients $H_{a,b,\chi}(n)$ may be defined without using asymptotic expansions of partial theta functions. To describe this situation, we require the Habiro ring, which was brilliantly studied by Habiro in \cite{Habiro}, and further connected to the $\mathbb F_1$ story and to the theory of Tate motives in \cite{Matilde}. This ring is defined as the completion
\[\mathcal H:=\lim_{\substack{\longleftarrow\\ n\geq0}}\Z[q]/((q;q)_n),\]
which may, as usual, be realized as the set of formal expansions of the form 
\[\sum_{n\geq0}a_n(q)(q;q)_n,\]
where $a_n\in\Z[q]$.
Associated to any element of the Habiro ring and to any root of unity $\zeta$ there is a power series expansion in $(\zeta-q)$ \cite{Habiro}. That is, there is a map 
\[\phi_\zeta\colon\mathcal H\rightarrow\Z[\zeta][[\zeta-q]].\]

\noindent
For example, the map $\phi_{1}$ may be realized explicitly (and be computed efficiently) for any $F=\sum_{n\geq0}a_n(q)(q;q)_n\in\mathcal H$ as the expansion $\sum_{n\geq0}c_n\left(1-q\right)^n$ by recursively solving for $c_n$ in the expansion
\[\sum_{n\geq0}a_n(1-q)(1-q;1-q)_n=\sum_{n\geq0}c_nq^n.\]
The resulting expressions defining $c_n$ always terminate, since $(1-q;1-q)_n=O\left(q^n\right)$. In particular, recalling the definition of $\xi(n)$, the Fishburn numbers are the coefficients of $\phi_1(F)$, where $F$ is Kontsevich's function.

A central result of Habiro (see Theorem 5.2 of \cite{Habiro}) states that the map $\phi_{\zeta}$ is injective for any $\zeta$, so that in fact, an element of the Habiro ring is uniquely determined by its power series expansion at any given root of unity. As Habiro points out, this is similar to the fact that a holomorphic function is determined by its power series expansion at a single point, so that $\mathcal H$ may be thought of as a ``ring of analytic functions on the roots of unity.''

Returning to the question of congruences, we now suppose that $P_{a,b,\chi}$ satisfies a ``strange identity'', i.e., that there is an $F_{a,b,\chi}\in\mathcal H$ such that 
\begin{equation}\label{StrangeAsymp}P_{a,b,\chi}\left(e^{-t}\right)\sim F_{a,b,\chi}\left(e^{-t}\right).\end{equation} Defining $c_{a,b,\chi}(n)$ as the coefficients of $\phi_1(F_{a,b,\chi})$, it directly follows from (\ref{StrangeAsymp}) and the definition of $H_{a,b,\chi}(n)$ that 
\[H_{a,b,\chi}(n)=c_{a,b,\chi}(n).\] 

Hence, we have given congruences for the coefficients of power series expansions of any element of the Habiro ring which satisfies a strange identity connecting it to a partial theta function $P_{a,b,\chi}$ satisfying the conditions of Theorem \ref{mainthm}. In Section \ref{ExampleS}, we study the half-derivatives of Andrews-Gordon functions, which Hikami showed in \cite{Hikami} satisfy strange identities. Furthermore, the resulting elements of the Habiro ring arise naturally in the context of Kashaev's invariant for the $(2m+1,2)$-torus knots and include Kontsevich's function as a special case.

More generally, it is very interesting to ask which partial theta functions satisfy strange identities, and what the general structure of such identities should be. For work related to this question, the interested reader is referred to important work of Coogan, Lovejoy, and Ono in \cite{CooganOno,LovejoyOno}, which studies connections between asymptotic expansions of partial theta series and $q$-hypergeometric series.

The paper is organized as follows. In Section \ref{Prelim}, we recall some standard facts on $p$-adic measures and $L$-functions, and we prove a fundamental result concerning congruences of polynomials in $L$-values, as well as a useful expression of the coefficients $H_{a,b,\chi}(n)$ in terms of $L$-values. Together, these results allow us to reduce the proof of Theorem \ref{mainthm} to an elementary statement on congruences for binomial coefficients, which we prove along with Corollary 1.2 in Section \ref{ProofMainThm}.
Finally, in Section \ref{ExampleS} we give an illuminating example of importance in knot theory, which also shows that Theorem \ref{mainthm} yields all known congruences for the Fishburn numbers, as well as new congruences for a class of sequences naturally generalizing the Fishburn numbers.

\section{Preliminaries}\label{Prelim}
\subsection{Congruences for polynomials in $L$-values}
In this subsection, we prove a useful theorem regarding the $p$-adic properties of certain polynomials in $L$-values. For any periodic sequence $\chi(n)$ of mean value zero, we define for $\operatorname{Re}(s)\gg0$ the $L$-function
\[L_{\chi}(s):=\sum_{n\geq1}\frac{\chi(n)}{n^s}.\]
This $L$-function has an analytic continuation to all of $\C$, which can easily be seen by writing $L$ as a linear combination of specializations of the Hurwitz zeta function and using the well-known continuation for the Hurwitz zeta function, which only has a simple pole of residue one at $s=1$. Thus, in particular, the values of $L(s)$ for $s\in-\N_0$ are well-defined.

In order to state our theorem, we first define a linear operator $\mathbb L_{\chi}\colon\Q[x]\rightarrow\C$ by its action on generators as 
\[\mathbb L_{\chi}\left(x^n\right):=L_{\chi}(-n),\]
and we recall that $f\in\Q[x]$ is called a \emph{numerical polynomial} on a set $X\subset\Z$ if $f(x)\in\Z$ for all $x\in X$. We also denote by $\mathrm{supp}(\chi)$ the support of $\chi$. Our main result concerning congruences of $L$-values is as follows.
\begin{theorem}\label{padic}
Let $\chi$ be a good function with period $M$, $p$ a prime with $(p,M)=1$, and $A\in\N$. If $f$ is a numerical polynomial on $\mathrm{supp}(\chi)$ such that for all $n\in\mathrm{supp}(\chi)$ we have
\[f(n)\equiv0\pmod{p^A},\]
then we also have
\[\mathbb L_{\chi}(f)\equiv0\pmod{p^A}.
\]
\end{theorem}
\begin{remark}
This theorem is a specialization of fairly standard facts on $p$-adic Dirichlet $L$-functions, or,
if one prefers, $p$-adic Mazur measures (see e.g. \cite{CourtieuPanchishkin}, 
\cite[Chapter 3]{Hida}). However, for the reader's convenience, we present a full and 
elementary proof below.
\end{remark}
\begin{proof}
We begin by recalling the definition of generalized Bernoulli numbers $B_{k,\chi}$, which are given by the generating function
\[
			\sum_{a=0}^{M-1} \chi(a) \frac{t e^{at}}{e^{Mt} - 1} =: \sum_{k \geq 0} B_{k,\chi} \frac{t^k}{k!}.
\]
As mentioned above, $L_{\chi}$ has an analytic continuation to all of $\C$. At non-positive integers, this is realized by the well-known relation
\[
L_\chi(-n) = -\frac{B_{n+1,\chi}}{n+1}.
\]	

Furthermore, it is possible  (see \cite[ChapterXIII, Theorem 1.2]{Lang}) to present these numbers as $p$-adic limits of related power sums:
\[
			B_{k,\chi} = \lim_{n \to \infty} \frac{1}{Mp^n} \sum_{a=0}^{Mp^n-1} \chi(a)a^k.
\]
We thus have that
\[
\bbL_\chi(x^k) = L_\chi(-k) = -\frac{B_{k+1,\chi}}{k+1}=-\lim_{n \rightarrow \infty}
\frac{1}{Mp^n} \sum_{a=0}^{Mp^n-1} \chi(a)\frac{a^{k+1}}{k+1}.
\]
Now let $c>1$ be a positive integer such that $(c,Mp)=1$, and define
\[
\bbL_\chi^c(f):=\bbL_\chi(f) - \bbL_{\chi^c}(f^c),
\] 
where the functions $f^c$ and $\chi^c$ are defined as $f^c(x):=f(cx)$ and $\chi^c(x):=\chi(cx)$.
We now write
\[
f(x)=\sum_{m=0}^Nd_m x^m
\]
with $d_m \in \bbQ$, and use the above to compute
\[
\bbL_\chi^c(f) = -\sum_{m=0}^Nd_m \lim_{n \rightarrow \infty}\frac{1}{Mp^n} \sum_{a=0}^{Mp^n-1} 
\left(\chi(a)\frac{a^{m+1}}{m+1} - \chi(ca)\frac{(ca)^{m+1}}{m+1} \right).
\]
Let $a_c$ be the element of $\{0,1, \ldots Mp^n-1\}$ such that $a_c\equiv ca \pmod{Mp^n}$.
Since $(c,Mp)=1$, multiplication by $c$ permutes residues modulo $Mp^n$, and we can rearrange the 
sum as
\[
\bbL_\chi^c(f) = -\sum_{m=0}^Nd_m \lim_{n \rightarrow \infty}\frac{1}{Mp^n} \sum_{a=0}^{Mp^n-1} 
\chi(ca)\left(\frac{a_c^{m+1} - (ca)^{m+1}}{m+1} \right).
\]
We now let $a_c=ca+Mp^nt_a$ with $t_a \in \bbZ$, make use of the congruence
\[
a_c^{m+1} - (ca)^{m+1} = (ca+Mp^nt_a)^{m+1} - (ca)^{m+1} \equiv (m+1)(ca)^mMp^nt_a \pmod{(Mp^n)^2},
\]
and reorder the summations to conclude that
\[
\bbL_\chi^c(f) = -\lim_{n \rightarrow \infty}\sum_{a=0}^{Mp^n-1} 
t_a\chi(ca)f(ca).
\]
In particular, we derive that for every $c>1$ such that $(c,Mp)=1$, 
\[
\bbL_\chi^c(f)  \equiv 0 \pmod{p^A}
\] 
whenever $f(n) \equiv 0 \pmod{p^A}$ for every $n \in \mathrm{supp}(\chi)$.

However, this does not suffice since we need the congruence for $\bbL_\chi(f)$ itself, not merely for the modified version $\bbL_\chi^c(f)  $. 
Since $\chi$ is a good function, it suffices to 
prove the congruence for the function
$\psi_{u,v}(n)$
where $u$ and $v$ satisfy the conditions of the definition of a good function.
We assume that $u$ and $v$ do satisfy these conditions for the remainder of the proof.
Now note that for an integer $c$ such that $(c,Mp)=1$ we have
\[
\psi_u^c(x) = \psi_u(cx)=\psi_{uc^{-1}(}x),
\]
where $cc^{-1} \equiv 1 \pmod M$.
We may now choose two integers $c_1>1$ and $c_2>1$ satisfying
\[
\begin{array}{c}
(c_1,Mp)=(c_2,Mp)=1\\
uc_1 \equiv v \pmod M \\
c_2 \equiv 1 \pmod M \\
c_1\equiv c_2  \pmod {p^C}
\end{array}
\]
for any large $C$, which we choose later.

We already know that 
\[
\bbL^{c_1}_{\psi_v}(f) = \bbL_{\psi_v}(f) - \bbL_{\psi_u}(f^{c_1}) \equiv 0 \pmod{p^A}
\]
and
\[
\bbL^{c_2}_{\psi_u}(f) = \bbL_{\psi_u}(f) - \bbL_{\psi_u}(f^{c_2}) \equiv 0 \pmod{p^A},
\]
and it suffices to prove that $C$ can be chosen big enough to guarantee the congruence 
\[
\bbL_{\psi_u}(f^{c_1})  - \bbL_{\psi_u}(f^{c_2}) = \bbL_{\psi_u}(f^{c_1}-f^{c_2}) \equiv 0 \pmod{p^A}.
\]
However, the assumptions on $u$ and $v$ imply that
\[
\bbL_{\psi_u}(f^{c_1}-f^{c_2} )= \sum_{m=0}^N d_m (c_1^m-c_2^m) \frac{B_{m+1,\psi_u}}{m+1}
\]
for some rational quantities $B_{m+1,\psi_u}/(m+1)$, which makes our claim obvious.

\end{proof}

\subsection{A useful formula for the sequence $H_{a,b,\chi}(n)$}
In this subsection, we relate the values of the sequences $H_{a,b,\chi}(n)$ to polynomials in $L$-values. Namely, we show the following result.
\begin{theorem}\label{xinumberlval}
Assume the conditions of Theorem \ref{mainthm}. Then we have 
\[H_{a,b,\chi}(n)=(-1)^n\mathbb L_\chi\left(x\binom{\frac{x^2-a^2}b}n\right).\]
\end{theorem}
\begin{proof}
Plugging in $q=e^{-t}$ into the definition of $P_{a,b,\chi}(q)$, we find
\[P_{a,b,\chi}\left(e^{-t}\right)=\sum_{n\geq1}n\chi(n)e^{-\frac{n^2-a^2}{b}t}.\]
Using a shifted version of the Euler-Maclaurin summation formula or a simple Mellin transform argument (see, for example, the proposition on page 98 of \cite{LawrenceZagier}), we find as $t\rightarrow0^+$ the asymptotic expansion 

\begin{equation}
\label{star}
\begin{aligned}
P_{a,b,\chi}\left(e^{-t}\right)&\sim e^{\frac{a^2}{b}t}\sum_{n\geq0}\frac{(-1)^nL_{\chi}(-2n-1)}{n!}\left(\frac{t}{b}\right)^n
\\ &=
\sum_{n\geq0}\frac{t^n}{n!b^n}\left(\sum_{j=0}^n(-1)^jL_{\chi}(-2j-1)a^{2(n-j)}\binom nj\right)
=\sum_{n\geq0}\frac{(-t)^n}{n!b^n}\mathbb L_{\chi}\left(x\left(x^2-a^2\right)^n\right).
\end{aligned}
\end{equation}

Now from the definition of the coefficients $H_{a,b,\chi}(n)$,  we have another expression for the asymptotic expansion of $P_{a,b,\chi}(e^{-t})$, given by 
\[P_{a,b,\chi}\left(e^{-t}\right)\sim\sum_{n\geq0}H_{a,b,\chi}(n)\left(1-e^{-t}\right)^n.\]
Letting $\left\{n\atop k\right\}$ denote the Stirling numbers of the second kind, which are defined by the generating series
\[\sum_{n\geq k}\frac{\left\{n\atop k\right\}x^n}{n!}:=\frac1{k!}\left(e^x-1\right)^k,\]
	we find
	
	\begin{equation}\label{doublestar}P_{a,b,\chi}\left(e^{-t}\right)\sim\sum_{n\geq0}\frac{(-t)^n}{n!}\left(\sum_{j=0}^n\left\{n\atop j\right\}j!H_{a,b,\chi}(j)(-1)^j\right).\end{equation}
By comparing coefficients of $t^n$ in (\ref{star}) and (\ref{doublestar}),
	we find
	
	\begin{equation}\label{xiformulainvert}
	\sum_{j=0}^n\left\{n\atop j\right\}j!H_{a,b,\chi}(j)(-1)^j=\frac{1}{b^n}\mathbb L_\chi\left(x\left(x^2-a^2\right)^n\right).
	\end{equation}
	
We now require the Stirling numbers of the first kind, which are defined by the relation
\[\sum_{j=0}^ns(n,j)x^j:=(x)_n,\]
where $(x)_n:=\prod_{j=0}^{n-1}(x-j)$ is the usual Pochhammer symbol.
Recall that the two types of Stirling numbers satisfy the inversion relationship
\begin{equation}\label{duality}\sum_{k=0}^ns(n,k)\left\{ k\atop j\right\}=\delta_{n,j},\end{equation}
where 
\[\delta_{x,y}:=\begin{cases}1&\text{ if }x=y,\\ 0&\text{ if }x\neq y.\end{cases}\]

Thus, if we have two sequences $u_n,v_n$ satisfying the relationship $v_n=\sum_{j=0}^n\left\{n\atop j\right\}u_j$, then we can easily invert to find 
\begin{equation}\label{inversion}
\sum_{k=0}^ns(n,k)v_k=\sum_{j=0}^nu_j\sum_{k=j}^ns(n,k)\left\{k\atop j\right\}=\sum_{j=0}^nu_j\sum_{k=0}^ns(n,k)\left\{k\atop j\right\}=\sum_{j=0}^nu_j\delta_{n,j}=u_n,
\end{equation}
where we used the fact that $\left\{k\atop j\right\}=0$ for $j>k$.

Combining (\ref{xiformulainvert}) and (\ref{inversion}), we obtain
	\[H_{a,b,\chi}(n)=\frac{(-1)^{n}}{n!}\mathbb L_\chi\left(x\sum_{j=0}^n\left(\frac{x^2-a^2}{b}\right)^js(n,j)\right)=\frac{(-1)^{n}}{n!}\mathbb L_{\chi}\left(x\left(\frac{x^2-a^2}{b}\right)_n\right),\]
which implies the desired formula for $H_{a,b,\chi}(n)$.

\end{proof}

\section{Proof of Theorem \ref{mainthm} and Corollary 1.2}\label{ProofMainThm}
We begin with a result giving a family of congruences for the sequences $H_{a,b,\chi}(n)$ which we will shortly see implies Theorem \ref{mainthm}. Firstly, however, we derive an elementary lemma on congruences of binomial coefficients. Specifically, we now recall Kummer's theorem, which allows us to easily study such congruences. 
\begin{theorem}[Kummer, \cite{Kummer}]\label{KummerThm}
Let $p$ be a prime, and suppose $n\in\Z$, $k\in\N$. Then the $p$-adic order of $\binom nk$ equals the number of carries when adding $k$ to $n-k$ in base $p$.
\end{theorem}
 From this, one can easily obtain the following lemma (see also Lemma 3.4 of \cite{Straub}).
\begin{lemma}\label{KummerCor}
Let $p$ be a prime, $s\in\{0,1,\ldots,p-1\}$, and $\alpha\in\N$. Then the following are true.
\begin{enumerate}
\item
If $B\in\{1,2\ldots,p-s-1\}$, then for all $A,n\in\mathbb N$,
\[\binom{s+p\alpha}{p^An-B}\equiv0\pmod{p^A}.\]
\item
If $B\in\{1,2\ldots,p-1\}$ and $\alpha\not\equiv-1\pmod p$, then for all $A,n\in\mathbb N$, 
\[\binom{s+p\alpha}{p^An-B}\equiv0\pmod{p^{A-1}}.\]
\end{enumerate}
\end{lemma}
\begin{proof}
Write $n:=s+p\alpha$ and $k:=p^An-B$, and denote the base $p$ coefficients of $n$ (resp. $k$) by $n_0,n_1,\ldots$ (resp. $k_0,k_1,\ldots$). By assumption, $n_0=s$ and $k_0=p-B$. Moreover, we have $k_1=k_2=\ldots=k_{A-1}=p-1$.
We now split into cases.

{\bf Proof of (1):} 
As $B\in\{1,2\ldots,p-s-1\}$, we have $k_0>n_0$. Denoting the base $p$ coefficients of $n-k$ by $m_0,m_1,\ldots$, we find that $m_0=p-k_0+n_0$, so that there is a carry when $m_0$ is added to $k_0$. Since $k_1=k_2=\ldots=k_{A-1}=p-1$, there are at least $A$ carries occur when $k$ is added to $n-k$, so by Theorem \ref{KummerThm}, we find the desired congruence.

{\bf Proof of (2):}
By part (1) of the Lemma, we may suppose that $B\in\{p-s,p-s+1,\ldots,p-1\}$, i.e., that $k_0\leq n_0$. Note that $m_0=n_0-k_0$, and by the assumption on $\alpha$, we find $n_1\neq p-1$, so that $m_1>0$. Hence, when adding $k$ to $n-k$, a carry occurs when adding $m_1$ to $k_1$. Since $k_2=k_3=\ldots=k_{A-1}=p-1$, there is also a carry when $m_i$ is added to $k_i$ for $i=2,\ldots A-1$, so that there are at least $A-1$ carries when $k$ is added to $n-k$. By Theorem \ref{KummerThm}, the result follows.
\end{proof}

In order to state our generalized version of Theorem \ref{mainthm}, we first set
\[S_{a,b,\chi,p}:=S:=\left\{s\in\N_0\colon s<p, \exists\ x\in\mathrm{supp}(\chi),\ x\not\equiv0\pmod p,\ \frac{x^2-a^2}b\equiv s\pmod p\right\},\]
and we define an analogous set 
\[S^*_{a,b,\chi,p}:=S^*:=\left\{s\in\N_0\colon s<p, \exists\ x\in\mathrm{supp}(\chi),\  \frac{x^2-a^2}b\equiv s\pmod p\right\}.\]
Then our main result is as follows.
\begin{theorem}\label{equivcondn}
Let $\chi$ be a good function and $p$ a prime, and suppose $a,b\in\Z$ are chosen so that $P_{a,b,\chi}\in\Z[[q]]$. Then the following are true.
\begin{enumerate}
\item If $B\in\{1,2,\ldots,p-1-\max_{s\in S^*}s\}$, then for all $n,A\in\N$, we have 
\[H_{a,b,\chi}\left(p^An-B\right)\equiv0\pmod{p^A}.\]
\item If $\beta\neq (p-1)$ and $(b,p)=1$, then for any $B\in\{1,2,\ldots,p-1-\max_{s\in S}s\}$ and for all $n,A\in\N$, we have
\[H_{a,b,\chi}\left(p^An-B\right)\equiv0\pmod{p^A}.\] 
\end{enumerate}
\end{theorem}
\begin{proof}
We begin by splitting into cases.

{\bf Proof of (1):}
By Theorem \ref{padic} and Theorem \ref{xinumberlval}, it suffices to show that for $B\in\{1,2,\ldots,p-1-\max_{s\in S^*}s\},$ we have
\[x\binom{\frac{x^2-a^2}{b}}{p^An-B}\equiv0\pmod{p^A}\] for all $A,n\in \mathbb N$, $x\in \mathrm{supp}(\chi)$. 
For fixed $x$, let $y:=\frac{x^2-a^2}{b}$, and let $s\in\mathbb N$ be the reduction of $y$ modulo $p$ with $0\leq s<p$. We claim that we can choose $\alpha\in\mathbb N$ so that 
\[\binom y{p^An-B}\equiv\binom{s+p\alpha}{p^An-B}\pmod{p^A}.\]
For this, it is enough to choose $\alpha$ satisfying
\[(y)_{p^An-B}\equiv(s+p\alpha)_{p^An-B}\pmod{p^{A+C}},\]
where $C:=\operatorname{ord}_p\left(\left(p^An-B\right)!\right)$. Clearly, it suffices to choose $\alpha$ with 
\[y\equiv s+p\alpha\pmod{p^{A+C}}.\] Now we can set $y-s=:zp,$ where $z\in\mathbb Z$. 
The last equation is then equivalent to
\[\alpha\equiv z\pmod{p^{A+C-1}},\]
which is clearly possible to solve for $\alpha$.
Thus, for such an $\alpha$, we have 
\[\binom y{p^An-B}\equiv\binom{s+p\alpha}{p^An-B}\pmod{p^A}.\]
Now (1) of Lemma \ref{KummerCor} directly implies that 
\[\binom{s+p\alpha}{p^An-B}\equiv0\pmod{p^A}\]
for all $B\in\{1,2,\ldots,p-1-\max_{s\in S*}s\}$, $n\in\N$, $x\in\mathrm{supp}(\chi)$.

{\bf Proof of (2):}
 To complete the proof, it suffices to show that if $x\equiv0\pmod p$, $\beta\neq (p-1)$, and $(b,p)=1$, then for all $0\leq B\leq p-1$ $n\in\N$, we have
 \[x\binom{\frac{x^2-a^2}{b}}{p^An-B}\equiv0\pmod p.\] 
 In this case, we find that $y\equiv -\frac{a^2}b\pmod{p^2}$. Hence, we may choose $\alpha$ as above, and assume without loss of generality that $y\equiv s+p\alpha\pmod{p^2}$. As $\beta\neq (p-1)$, we have that $\alpha\not\equiv-1\pmod p$. By
 (2) of Lemma \ref{KummerCor}, we find that 
\[\binom{s+p\alpha}{p^An-B}\equiv0\pmod{p^{A-1}}.\]
Hence,
\[x\binom{\frac{x^2-a^2}{b}}{p^An-B}\equiv x\binom{s+p\alpha}{p^An-B}\equiv0\pmod{p^A},\]
as desired.

\end{proof}
We are now in a position to prove Theorem \ref{mainthm}. First, consider the sets
\[T_{a,b,\chi,p}:=T:=\left\{s\in\N_0\colon s<p, \exists\ x\in\Z,\ x\not\equiv0\pmod p,\ \frac{x^2-a^2}b\equiv s\pmod p\right\},\]
\[T^*_{a,b,\chi,p}:=T^*:=\left\{s\in\N_0\colon s<p, \exists\ x\in\Z,\ \frac{x^2-a^2}b\equiv s\pmod p\right\}.\]
Clearly, since $S\subset T$, and $S^*\subset T^*$ if $B$ satisfies the conditions of Theorem 3.3 with $S$ replaced by $T$ or with $S^*$ replaced by $T^*$, then the same congruence for $H_{a,b,\chi}$ holds. Theorem \ref{mainthm} follows from this observation, together 
with the following elementary result, whose proof follows from a straightforward calculation.
\begin{lemma}\label{equivalentcondns}
Let $a,b\in\Z$, let $p$ be a prime and suppose $(p,b)=1$. Then the following are equivalent conditions for $B\in\N$.
\begin{enumerate}
\item We have that $B$ satisfies
\[
			\left(\frac{a^2-b}{p}\right) =-1, \ \left(\frac{a^2-2b}{p}\right) =-1, \dots, \ \left(\frac{a^2-Bb}{p}\right) =-1.
		\]
\item We have that
\[B\in\{1,2,\ldots,p-1-\max_{t\in T^*}t\}.\]

\end{enumerate}
Moreover, the following are also equivalent conditions for $B\in\N$.

\begin{enumerate}
\item We have that $B$ satisfies
\[
			\left(\frac{a^2-b}{p}\right) \neq 1, \ \left(\frac{a^2-2b}{p}\right) \neq 1, \dots, \ \left(\frac{a^2-Bb}{p}\right) \neq 1.
		\]
\item We have that
\[B\in\{1,2,\ldots,p-1-\max_{t\in T}t\}.\]
\end{enumerate}

\end{lemma}
\begin{remark}
An elementary argument shifting integers $x$ by multiples of $p$ shows that if $(p,M)=1$, then we have $S=T$ and $S^*=T^*$. Hence, for all but finitely many primes $p$, the congruences given in Theorem 3.3 are exactly the same as the congruences implied by Theorem 1.1.
\end{remark}

We now prove Corollary 1.2.
\begin{proof}[Proof of Corollary 1.2]
It suffices to check that if $a^2-b$ is not a square and $a^2-b\not\equiv3\pmod4$, then exactly $50\%$ of primes $p$ satisfy $\left(\frac{a^2-b}p\right)=-1$, since Theorem 1.1 then implies that $H_{a,b,\chi}(pn-1)\equiv0\pmod p$ for all $n\in\N$. To see that this is the case, note that given the conditions on $a^2-b$, the function $\left(\frac{a^2-b}\cdot\right)$ is a non-principal Dirichlet character, and hence takes on values $1$ and $-1$ equally often (as the sum over a complete set of representatives of residue classes modulo the modulus of the character is zero). Moreover, the values where this character are non-zero are exactly those values which are coprime to the modulus of the character. Hence, by the Chebatorev density theorem applied to primes in arithmetic progressions, exactly $50\%$ of primes satisfy the desired condition.
\end{proof}

\section{Fishburn numbers and Hikami's functions}\label{ExampleS}
In this section, we work out a particularly important family of examples of Theorem \ref{mainthm}. Specifically, we consider a collection of quantum modular forms whose beautiful properties were laid out by Hikami in \cite{Hikami}. For further important results ``inverting'' these functions and relating them to indefinite theta series and mock theta functions, see also \cite{BrysonOnoPitmanRhoades,HikamiLovejoy}. The quantum modular forms defined by Hikami are then given for $m\in\mathbb N$  and $\alpha\in\{1,2,\ldots,m-1\}$ by
	\[F_m^{(\alpha)}(q):=\sum_{k_1,k_2,\ldots,k_m=0}^{\infty}(q;q)_{k_m}q^{k_1^2+\ldots+k_{m-1}^2+k_\alpha+\ldots+k_{m-1}}\left(\prod_{\substack{i=1\\ i\neq \alpha}}^{m-1}\left[{k_i+1}\atop{k_i}\right]_q\right)\cdot\left[{k_{\alpha+1}+1}\atop{k_\alpha}\right]_q,\]
	where the usual $q$-binomial is defined by 
	\[\left[ n\atop k\right]_q:=\begin{cases}\frac{(q;q)_n}{(q;q)_{k}(q;q)_{n-k}}&\text{ if }0\leq k\leq n,\\ 0&\text{ otherwise}.\end{cases}\]
	We note that the function $F_1^{(0)}(q)$ reduces simply to Kontsevich's function $F(q)$. Moreover, it is clear from the definition that $F_m^{(\alpha)}\in\mathcal H$. The connection to partial theta functions is shown in (15) of \cite{Hikami}, which states 	that $F_m^{(\alpha)}$ shares an asymptotic expansion with a half-derivative of an Andrews-Gordon function. Specifically, Hikami shows that there is a strange identity connecting $F_m^{(\alpha)}$ and 
	\begin{equation}\label{HikamiStrange}-\frac12\sum_{n\geq1}n\chi_{8m+4}^{(\alpha)}(n)q^{\frac{n^2-(2m-2\alpha-1)^2}{8(2m+1)}},\end{equation}
	where 
	\[\chi_{8m+4}^{(\alpha)}(n):=\begin{cases}1&\text{ if }n\equiv 2m-2\alpha-1\pmod{8m+4},\\ -1&\text{ if }n\equiv 2m+2\alpha+3\pmod{8m+4},\\-1& \text{ if }n\equiv 6m-2\alpha+1\pmod{8m+4},\\ 1 & \text{ if }n\equiv 6m+2\alpha+5\pmod{8m+4},\\ 0&\text{ otherwise}.\end{cases}\]
We now note that $P_{2m-2\alpha-1,8(2m+1),\chi_{8m+4}^{(\alpha)}}\in\Z[[q]]$, since, as in (10) of \cite{Hikami}, 
\[\frac{1}{(q;q)_{\infty}}\sum_{n\geq0}\chi_{8m+4}^{(\alpha)}(n)q^{\frac{n^2-(2m-2\alpha-1)^2}{8(2m+1)}}=\prod_{\substack{n\geq1\\ n\not\equiv0,\pm(\alpha+1)\pmod{2m+1}}}\left(1-q^n\right)^{-1}.\]
Moreover, by the discussion in Section 1, in this case, the coefficients $H_{2m-2\alpha-1,8(2m+1),\chi_{8m+4}^{(\alpha)}}$ may be defined combinatorially as the coefficients of the expansion
		\[\sum_{n\geq0}\xi_{m}^{(\alpha)}(n)q^n:=F_m^{(\alpha)}(1-q),\]
		so that
		\[H_{2m-2\alpha-1,8(2m+1),\chi_{8m+4}^{(\alpha)}}(n)=\xi_{m}^{(\alpha)}(n).\]
	As an example, 
	consider 
	\[F_2^{(0)}(q)=\sum_{n\geq0}(q;q)_n\sum_{k=0}^nq^{k(k+1)}\left[ n\atop k\right]_q,\]
	so that the first few coefficients $\xi_2^{(0)}(n)$ are given by $1,2,6,23,109\ldots$. Numerical calculations suggest that the following congruences hold for all $n,A\in\N$:
	\[\xi_2^{(0)}\left(3^An-1\right)\equiv0\pmod{3^A},\]
	\[\xi_2^{(0)}\left(11^An-a\right)\equiv0\pmod{11^A}, \text{ where }a\in\{1,2,3\},\]
	\[\xi_2^{(0)}\left(13^An-a\right)\equiv0\pmod{13^A}, \text{ where }a\in\{1,2,3,4\}.\]
These congruences follow immediately from Theorem \ref{mainthm} and Lemma 3.4, along with a short computation, once the following result is checked.
\begin{lemma}
For any $m\in\N$, $a\in\{1,2,\ldots,m-2\}$, $\chi_{8m+4}^{(\alpha)}(n)$ is a good function. 
\end{lemma}  
\begin{proof}
Let $x:=2m+1$ and $y:=\alpha+1$. Then the period of $\chi_{8m+4}^{(\alpha)}$ is $M=4x$, and by the definition of the sequence, it is clearly supported on odd integers. Furthermore, one easily checks that 
\[\chi_{8m+4}^{(\alpha)}=\psi_{x-2y,x+2y}-\psi_{3x-2y,3x+2y}.\]
Now it is easy to verify that
\[(4x,x-2y)=(4x,x+2y)=(x,y)\]
and 
\[(4x,3x-2y)=(4x,3x+2y)=(x,y).\] 
Hence, $\psi_{x-2y,x+2y}$ and $\psi_{3x-2y,3x+2y}$ both satisfy the conditions of a good function, so that $\chi_{8m+4}^{(\alpha)}$ is good as well.
\end{proof}
Thus, we may apply Theorem \ref{mainthm} to the coefficients $\xi_{m}^{(\alpha)}$, which directly implies congruences for $\xi_m^{(\alpha)}(n)$.
Using Corollary 1.2 and the fact that 
\[(2m-2\alpha-1)^2-8(2m+1)\equiv0,1\pmod4,\] we immediately deduce the following result.
\begin{corollary}\label{CorCor}
Choose $\alpha,m\in\N$ with $\alpha<m$ such that $(2m-2\alpha-1)^2-8(2m+1)$ is not a square. Then 
\[\xi_m^{(\alpha)}\left(p^An-1\right)\equiv0\pmod{p^A}\]
for all $n,A\in\N$ for at least $50\%$ of primes $p$.
\end{corollary}
\begin{remark}
The condition that $(2m-2\alpha-1)^2-8(2m+1)$ is not a square is necessary in Corollary \ref{CorCor}. For example, if $\alpha=1$, $m=7$, then we have $a^2=121$, $b=120$. Hence, Theorem \ref{mainthm} yields a congruence modulo $p$ only when $\left(\frac{a^2-b}p\right)=\left(\frac1p\right)\neq1$, which does not hold for any $p$.
\end{remark}

We now take a closer look at the congruences for $\xi(n)$, which inspired this paper. Using the well-known fact that for $p\geq5$, $\frac{p^2-1}{24}\in\Z$, we find that if $\frac{p^2-1}{24}\equiv s\pmod p$ with $0\leq s<p$, then 
\[\beta_p=\frac{\frac{p^2-1}{24}-s}{p}=\left\lfloor\frac {p^2-1}{24p}\right\rfloor\neq (p-1).\] By Theorem \ref{mainthm}, we find that 
\[\xi\left(p^An-1\right)\equiv0\pmod{p^A}\]
for all $A,n\in\N$ whenever
\[\left(\frac{-23}{p}\right) \neq 1.\]
Using quadratic reciprocity, it is easy to check that this occurs exactly when 
\[p=23\text{ or }p\equiv 5,7,10,11,14,15,17,19,20,21,22\pmod {23},\]
which is remarked in \cite{AndrewsSellers} and  \cite{Straub}.

Finally, we show that our set of congruences for $\xi(n)$ is the same as that which was given in \cite{Straub}. Indeed, one easily finds by an elementary argument that the set $S$ in Theorem 3.3 may be replaced by the set $T$ (since $(p,M)=1$). Comparing with Theorem 1.2 of \cite{Straub}, it suffices to show that the set of reductions of values of $\frac{x^2-1}{24}$ modulo $p$ as $x$ ranges over $\Z\setminus p\Z$ is equal to the set of reductions of pentagonal numbers $\frac12x(3x-1)$ modulo $p$ with $x\not\equiv0\pmod p$ for any prime $p\geq5$. 
However, this follows immediately by noting that $\frac{x^2-1}{24}$ becomes $\frac12x(3x-1)$ upon substituting $x$ with $6x-1$, which simply permutes the residue classes of $x$ modulo $p$. Thus, thanks to the extensive calculations of Straub \cite{Straub}, we have conjecturally given all linear congruences for $\xi(n)$. That is, following Andrews and Garvan and Straub, we make the following conjecture, which we leave as an important challenge for future work.
\begin{conjecture}
Let  $p$ be a prime. Then there exists a $B\in\N$ such that there is a congruence 
\[\xi\left(p^An-B\right)\equiv0\pmod{p^A}\]
for all $n$ precisely when 
\[p=23\text{ or }p\equiv 5,7,10,11,14,15,17,19,20,21,22\pmod {23}.\]
\end{conjecture}
We conclude by noting that there are other congruences for linear combinations of Fishburn numbers. Specifically, in unpublished work, Garthwaite and Rhoades observed that for all $n\in\N$, we have
\[\xi(5n+2)-2\xi(5n+1)\equiv0\pmod5,\]
\[\xi(11n+7)-3\xi(11n+4)+2\xi(11n+3)\equiv0\pmod{11}.\]
In Theorem 1.3 of \cite{Garvan}, Garvan established an infinite set of congruences which includes this example.
It is likely that the methods of this paper extend to prove these congruences of Garvan, and that similar congruences hold for other quantum modular forms. We leave the details to the interested reader.

\end{document}